\documentclass[11pt,a4paper,oneside,reqno]{amsart}

\usepackage{changepage}
\usepackage[utf8]{inputenc}
\usepackage{graphicx}
\usepackage{amsmath,amsfonts,amsthm,amssymb,dsfont}
\usepackage{mathrsfs,mathtools,mathabx}
\usepackage{verbatim}
\usepackage{enumerate}
\usepackage{enumitem}
\usepackage{xcolor}
\usepackage{cite}
\usepackage{array}
\usepackage[T1]{fontenc}
\usepackage{tikz}
\usepackage{subcaption}

\usepackage[
    colorlinks,
    citecolor=blue,
    linkcolor=red!60!black,
    pagebackref=true
]{hyperref}
\usepackage[nameinlink]{cleveref}

\crefname{ineq}{inequality}{inequalities}
\crefname{interv}{intervals}{intervals}
\creflabelformat{ineq}{#2{\upshape(#1)}#3}
\creflabelformat{interv}{#2{\upshape(#1)}#3}

\newtheorem{theorem}{Theorem}

\newtheorem{lemma}[theorem]{Lemma}
\newtheorem{proposition}[theorem]{Proposition}
\newtheorem{corollary}[theorem]{Corollary}

\newtheorem{remark}[theorem]{Remark}

\newcommand{\A}{\mathcal{A}}

\newcommand{\bigO}{\mathcal{O}}
\newcommand{\smallo}{o}
\newcommand{\Z}{\mathbb{Z}}
\newcommand{\N}{\mathbb{N}}
\newcommand{\Q}{\mathbb{Q}}
\newcommand{\R}{\mathbb{R}}

\newcommand{\T}{\mathbb{T}}

\newcommand{\Leb}{\mathrm{Leb}}

\title{Coboundaries of 3-{IET}s}

\author[Berk]{Przemysław Berk}
\address{Faculty of Mathematics and Computer Science, Nicolaus Copernicus University, ul. Chopina 12/18, 87-100 Toruń}
\email{zimowy@mat.umk.pl}

\author[Ospina]{Carlos Ospina}
\address{School of Mathematical Sciences, Tel Aviv University, Tel Aviv, 69978, Israel}
\email{ospina.math@icloud.com}

\date{}

\begin{document}

\begin{abstract}
\textcolor{black}{In this note, we investigate coboundaries of interval exchange transformations of three intervals, or 3-{IET}s. More precisely, we study differentiable functions whose derivative is absolutely continuous, and whose integral is zero and whose derivative also has integral zero. We show that these functions are coboundaries for a typical 3-{IET} if and only if their values at the endpoints of the domain are zero. We also show the existence of exceptional counterexamples for both possible endpoint behaviors. Our results are obtained by studying the properties of associated skew products.}
\end{abstract}

\maketitle

\section{Introduction}
One of the classical problems in dynamical systems is to determine whether, 
for a given measure-preserving transformation $T$ on a standard
\textcolor{black}{Borel probability} space $(X,\mu)$, 
a given zero-mean function $f\in L^1(X,\mu)$ is a coboundary. 
That is, one asks whether there exists $g\in L^1(X,\mu)$, 
possibly satisfying additional regularity conditions, such that the cohomological equation 
\begin{equation}\label{eq: cohom}
f(x)=g\circ T(x)-g(x)
\end{equation}
holds for $\mu$-almost every $x$.

This problem has been studied from many points of view and for many classes of dynamical systems. In particular, interval exchange transformations were studied in \cite{MarmiMoussaYoccoz} and \cite{BerkTrujillo}, while \cite{ForniFlows} studied the analogous problem for flows on compact surfaces. The main application of this paper is  to use the cohomological equation to determine whether a skew product $T_f$ on $X\times \R$, given by
\[
T_f(x,r):= (Tx,r+f(x))
\]
is ergodic with respect to the product measure $\mu \otimes \Leb_{\R}$. 
\textcolor{black}{%
It is a classical fact that if $f$ is a coboundary with $g$ as in \Cref{eq: cohom}, then, for $\mu$-almost every $x$ and every $r\in\R$,
 $(x,r)$ and its trajectory under $T_f$ is contained in the vertically shifted graph
\[
G_{r,x}:=\{(y,g(y)+r-g(x)):y\in X\}.
\]
For fixed $-\infty<a<b<\infty$, the set
\[
G_{[a,b],x}:=\bigcup_{s\in[a,b]}G_{s,x}
\]
has positive measure, and neither $G_{[a,b],x}$ nor its complement has full $\mu\otimes\Leb_\R$-measure.
}
Moreover, $G_{[a,b]}$ is $T_f$-invariant. Consequently, $T_f$ is not ergodic. Conversely, if one can prove that $T_f$ is ergodic, then $f$ is not a coboundary. 

We study this problem for the following setting: $X = [0,1)$ with Lebesgue measure $\Leb$, $T$ is a symmetric interval exchange transformation of $3$ intervals (see  \Cref{sec:iets} for definitions) and $f\in C^{1+AC}([0,1))$, where $C^{1+AC}(\cdot)$ denotes the set of differentiable functions whose derivative is absolutely continuous. We consider two cases, depending on the values of $f$ at the endpoints of the interval. First, we assume that $f$ vanishes at the endpoints of the interval $[0,1)$.

\begin{theorem}\label{thm:main1}
\textcolor{black}{Let $f\in C^{1+AC}([0,1))$ satisfy
    \[
        f(0)=\lim_{x\to 1}f(x)=0
        \qquad\text{and}\qquad
        \int f=0.
    \]
}
Then $f$ is a coboundary for almost every symmetric $3$-IET $T$. In particular, the skew product $T_f$ is not ergodic.

Assume in addition that $f'$ is not identically zero and that
\[
f'(0)=\lim_{x\to 1}f'(x)\neq 0.
\]
Then, among the $3$-IETs $T$ whose discontinuities are not zeros of \textcolor{black}{$f'$}, there is a dense subset for which the skew product $T_f$ is ergodic.
\end{theorem}

In contrast, we have the following result when \(f\) has equal nonzero endpoint values.

\begin{theorem}\label{thm:main2}
    Let \textcolor{black}{$f\in C^{1+AC}([0,1))$ satisfy
\[
\int f=0
\qquad\text{and}\qquad
f(0)=\lim_{x\to 1}f(x)\neq 0.
\]}
Then the skew product $T_f$ is ergodic for almost every symmetric $3$-IET $T$. In particular, $f$ is not a coboundary for $T$.
\end{theorem}

\textcolor{black}{
Similarly to \Cref{thm:main1},
} 
there is also a small family of maps whose behavior differs from the typical case.

\begin{proposition}\label{prop:main3}
\textcolor{black}{
Let $f\in C^{1+AC}([0,1))$ satisfy
\[
\int f=0
\qquad\text{and}\qquad
f(0)=\lim_{x\to 1}f(x)\neq 0.
\]
}
Then there exists a dense family of minimal $3$-IETs $T$ for which $f$ is a coboundary. In particular, the skew product $T_f$ is not ergodic.
\end{proposition}

Note that in all the results stated above, we assumed that
\[
f(0)=\lim_{x\to 1}f(x),
\]
or equivalently,
\[
\int_0^1 f'(x)\,dx=0.
\]
The case
\[
\int_0^1 f'(x)\,dx\neq 0
\]
is also interesting and was treated by Pask \textcolor{black}{for rotations} in \cite[Theorem 2.11]{Pask}; see \Cref{prop:Pask}. \textcolor{black}{For IETs, such functions were treated by Trujillo, Wu and the first-named author in \cite{BerkTrujilloWu}. Although they were able to consider symmetric IETs with an arbitrary number of intervals, they needed an additional \emph{antisymmetry} assumption on $f$ in order to prove ergodicity of the associated skew products.}

It is worth noting that both sine and cosine are coboundaries for irrational rotations. Indeed, let $\alpha$ be irrational and let $R(x):=x+\alpha \bmod{1}$. We first show this for $f(x)=\sin(2\pi x)$. By the difference-of-cosines identity,
\[
\begin{split}
\sin(2\pi x)
&=
\frac{\cos(2\pi (x+\alpha))-\cos(2\pi(x-\alpha))}{-2\sin(2\pi\alpha)} \\
&=
\frac{\cos(2\pi (x+\alpha))+\cos(2\pi x)}{-2\sin(2\pi\alpha)}
-
\frac{\cos(2\pi x)+\cos(2\pi(x-\alpha))}{-2\sin(2\pi\alpha)}.
\end{split}
\]
Since $\alpha$ is irrational, $\sin(2\pi\alpha)\neq 0$. Therefore, if
\[
g(x):=
\frac{\cos(2\pi x)+\cos(2\pi(x-\alpha))}{-2\sin(2\pi\alpha)},
\qquad x\in\R/\Z,
\]
then
\[
\sin(2\pi x)=g(Rx)-g(x).
\]
Thus $\sin(2\pi x)$ is a coboundary for $R$. Since $g$ is differentiable, differentiating both sides gives
\[
2\pi\cos(2\pi x)=g'(Rx)-g'(x).
\]
Hence $\cos(2\pi x)$ is also a coboundary, \textcolor{black}{with transfer function $\frac{1}{2\pi}g'(x)$.} This contrasts with \Cref{thm:main2}.

Our result is directly related to a question of Chaika and Robertson \cite{ChaikaRobertson}. In \cite[Question 1.5]{ChaikaRobertson}, they ask whether, for a typical IET on at least three intervals, the skew product is ergodic \textcolor{black}{for the cocycle $f(x)=\cos(2\pi x)$.} Thus, \Cref{thm:main2} gives an affirmative answer to their question in the case of symmetric $3$-IETs.

\begin{corollary}
For almost every symmetric $3$-IET $T$ on the interval $[0,1)$ and $f(x)=\cos(2\pi x)$, the skew product $T_f$ is ergodic.
\end{corollary}

In contrast, \Cref{thm:main1} gives the following result.

\begin{corollary}
For almost every symmetric $3$-IET $T$ on the interval $[0,1)$ and $f(x)=\sin(2\pi x)$, the skew product $T_f$ is not ergodic. However, there exists a dense subset of \textcolor{black}{$3$-IETs} for which $T_f$ is ergodic.
\end{corollary}

In \Cref{sec:iets}, we recall basic definitions \textcolor{black}{and facts concerning IETs}, rotations and skew products. In \Cref{sec:skprod}, we state the essential-value criterion for the ergodicity of skew products (\Cref{thm:FUcriterion}). We then turn to the proofs of the main theorems. In \Cref{sec:mainproof}, we outline the proof of \Cref{thm:main1}, using results of \cite{MarmiMoussaYoccoz} and \cite{Krzysztof2000}. Finally, in \Cref{sec:mainproof2}, we prove \Cref{thm:main2} using the classical Denjoy-Koksma inequality, and we prove \Cref{prop:main3} at the end of the paper.

\textcolor{black}{We also point out that the question about the cosine cocycle posed by Chaika and Robertson in \cite{ChaikaRobertson} concerns IETs with an arbitrary number of intervals and remains open in that generality. 
The argument in the present paper uses a feature special to symmetric $3$-IETs: after inducing on a suitable subinterval, the induced transformation is a rotation.
For typical IETs with more intervals, any induced map is again an IET with more than two exchanged intervals, rather than a rotation. Thus the reduction used here is not available in that setting, except in degenerate combinatorial situations.}

\section*{Acknowledgments}
The authors thank Krzysztof Frączek for suggesting the problem and Jon Chaika for his continuous support. The authors are also grateful to the anonymous referee for their careful reading of the manuscript and for their helpful comments and suggestions. This research was partially supported by the Narodowe Centrum Nauki Grant 2022/45/B/ST1/00179.

\section{Rotations and interval exchange transformations}\label{sec:iets}

Given $\alpha \in \R$, we denote by $R:[0,1)\to[0,1)$ the map
\[
x \mapsto x+\alpha \mod 1.
\] 
This map is called the rotation by $\alpha$ on $[0,1)$. Rotations are special cases of interval exchange transformations, or IETs. For more information about this class of maps, we refer the reader to \cite{Viana2006,claynotes}.

Let $\A$ be an alphabet of $d$ symbols, and let $\lambda \in \R_+^\A$ be a vector of positive entries. Given two bijective functions $\pi_0:\A\to\{1,\dots,d\}$ and $\pi_1:\A\to\{1,\dots,d\}$, let  $\pi$ be the permutation of symbols in $\A$ defined by
\begin{equation*}
    \pi=\begin{pmatrix}
        \pi_0^{-1}(1) & \dots & \pi_0^{-1}(d)\\
        \pi_1^{-1}(1) & \dots & \pi_1^{-1}(d)\\
    \end{pmatrix}.
\end{equation*}
Let $|\lambda| = \sum_{a \in\A}\lambda_a$, $I = [0,|\lambda|)$ and let $I_a \subset I$ be the interval 
\[
\textcolor{black}{%
\left[
\sum_{\{b:\pi_0(b)<\pi_0(a)\}}\lambda_b,
 \sum_{\{b:\pi_0(b)\leq \pi_0(a)\}}\lambda_b
\right)
.}
\]
The vector $\lambda$ and the permutation $\pi$ determine the interval exchange transformation on $d$ intervals $T_{\lambda,\pi}:I \to I$ given by
\[
\textcolor{black}{T_{\lambda,\pi}(x) = x +
\sum_{\{b:\pi_1(b)<\pi_1(a)\}}\lambda_b
-
\sum_{\{b : \pi_0(b) < \pi_0(a)\}}\lambda_b,\quad \text{if }x \in I_a.}
\]
We will write $T$ instead of $T_{\lambda,\pi}$ if $\lambda$ and $\pi$ are given. Also, we say $d$-IET instead of ``IET in $d$ intervals''.

Given a fixed permutation $\pi$, the length parameter space $\R^\mathcal A_+$ determines a family of $d$-IETs. This set is endowed with its natural topology and Lebesgue measure. Hence, generic properties such as density or full measure of IETs refer to properties of the corresponding subset of the parameter space $\R^\mathcal A_+$.

The (possible) discontinuities of $T_{\lambda,\pi}$ are of the form
$$\xi_{\beta} = \sum_{\pi_0(\gamma)<\pi_0(\beta)}\lambda_\gamma$$
and we denote the set of discontinuities by
$D(T_{\lambda,\pi})=\{\xi_\beta: \beta \in \mathcal{A}\}$. 

{An IET $T=T_{\lambda,\pi}$ is said to have the \textbf{Keane property} if for all $m \in \N$ and for all pairs of symbols $\beta, \gamma \in \mathcal{A}$,
$$T^m\xi_\beta \neq \xi_\gamma.$$}

This work focuses on 3-IETs where the permutation is of the form 
\begin{equation}
    \pi=\begin{pmatrix}
    A & B & C \\
    C & B & A
\end{pmatrix},
\label{SymPerm}
\end{equation}
and $2$-IETs  (rotations) with 
\begin{equation*}
    \pi=\begin{pmatrix}
         A & B  \\ B & A
        \end{pmatrix}.
\end{equation*}

\subsection{Induced maps of 3-IETs are rotations} \label{sec:reduction_to_rotation}

Let $T=T_{\lambda,\pi}$ be an IET and $f:[0,|\lambda|) \to \R$ be a real-valued function. We denote \textcolor{black}{by}
\[
T_f : [0,|\lambda|) \times \R \to  [0,|\lambda|) \times \R
\]
the skew product map
\[
(x,t) \mapsto (Tx,t+f(x)).
\]

For every fixed $3$-IET $T$, we consider the first return map to a suitable subinterval $J\subset[0,|\lambda|)$. The subinterval $J$ is either $[0,\lambda_A+\lambda_B)$ or $[0,\lambda_B+\lambda_C)$, depending on which of the two cases below holds. This is the classical Rauzy--Veech induction step; see, for example, \cite{Viana2006} for more details.

Denote by
\[
R_{\widetilde f}:J\times\R\to J\times\R
\]
the first return map of $T_f$ to the cylinder $J\times\R$. Then
\[
R_{\widetilde f}(x,t)=\left(T|_Jx,t+\widetilde f(x)\right),
\]
where $\widetilde f$ is the induced cocycle.

Let $\lambda = (\lambda_A,\lambda_B,\lambda_C)$ be a vector of positive real entries and let $\pi$ be the permutation in \Cref{SymPerm}. 

{\bf Case 1: $\lambda_A > \lambda_C$.} Let $T|_J$ be the first return map of $T$ to $J=[0,\lambda_A+\lambda_B)$. A direct computation shows that $T|_J$ is a rotation by $\alpha = \lambda_B+\lambda_C$ modulo $|J| = \lambda_A+\lambda_B$. The new cocycle function for the first return map is defined by
\begin{equation}\label{eq:case1}\widetilde{f}(x) = 
\left\{
\begin{array}{cll}
   f(x)  &\text{ if }& 0\le x <\lambda_A-\lambda_C \\
   f(x) + f(x+\lambda_B+\lambda_C)  &\text{ if }& \lambda_A - \lambda_C \le x < \lambda_A\\
   f(x) &\text{ if }& \lambda_A \le x <\lambda_A+\lambda_B
\end{array}
\right.
\end{equation}

{\bf Case 2:} $\lambda_A < \lambda_C$. Let $J=[0,\lambda_B+\lambda_C)$ be the interval of induction. The first return map of $T$ to $J$ is the rotation by $\alpha = \lambda_C-\lambda_A$ modulo $|J|=\lambda_B+\lambda_C$. Similarly, the cocycle function for the first return map is the function
\begin{equation}\label{eq:case2}
    \widetilde{f}(x) = \left\{
    \begin{array}{cll}
        f(x) + f(x+\lambda_B+\lambda_C) &\text{ if }&  0 \leq x < \lambda_A\\
        f(x) &\text{ if }& \lambda_A \leq x <\lambda_B+\lambda_C
    \end{array}
    \right.
\end{equation}

\begin{remark}
\label{rmk:ftildeisabscont}
Let $f$ be absolutely continuous and $f(0) = \lim_{x \to |\lambda|} f(x) = 0$.\par 

\noindent In Case 1, when \Cref{eq:case1} applies, \textcolor{black}{if $f$ does not vanish at the relevant discontinuity of $T$, then $\widetilde{f}$ has a nonzero jump}. Nevertheless, it is piecewise absolutely continuous. The intervals of absolute continuity are $[0,\lambda_A-\lambda_C)$ and $[\lambda_A - \lambda_C, \lambda_A+\lambda_B)$. These two intervals are exchanged by the rotation $T|_J$.\par
    
\noindent In Case 2, when \Cref{eq:case2} applies, the induced cocycle $\widetilde{f}$ is absolutely continuous on $J$.
\end{remark}

To show that a cocycle function is a coboundary, we often prove that the induced cocycle function is a coboundary rather than directly working with the original cocycle. This is a standard observation, but we include the proof in the notation used in this paper.

\begin{lemma}\label{rmk:inducedenough}
Let $(X,\mathcal B,\nu,T)$ be an invertible ergodic measure-preserving system, and let $Y\subset X$ be a measurable set with $\nu(Y)>0$. Let $T_Y:Y\to Y$ be the first return map to $Y$, and let $r:Y\to\mathbb N$ be the first return time. For a measurable function $f:X\to\R$, denote by $\widehat{f}$ the induced cocycle
\[
    \widehat{f}(x)=\sum_{j=0}^{r(x)-1} f(T^jx),
    \qquad x\in Y.
\]

If $\widehat f$ is a coboundary for $T_Y$, then $f$ is a coboundary for $T$ on the full-measure set
\[
    \bigcup_{n=0}^{\infty}T^nY.
\]
In particular, since $T$ is ergodic, $f$ is a coboundary for $T$ almost everywhere.
\end{lemma}

\begin{proof}
Suppose that there exists a measurable function $\widehat h:Y\to\R$ such that
\[
    \widehat f(x)=\widehat h(T_Yx)-\widehat h(x).
\]
For $x\in Y$ and $j=0,\dots,r(x)-1$, define
\[
    h(T^jx):=
    \begin{cases}
        \widehat h(x), & j=0,\\
        h(T^{j-1}x)+f(T^{j-1}x), & j\ge 1.
    \end{cases}
\]
Then $f(y)=h(Ty)-h(y)$ for every
\[
    y\in \bigcup_{n=0}^{\infty}T^nY.
\]
Since $T$ is ergodic and $\nu(Y)>0$, this set has full measure. Therefore the cohomological equation holds almost everywhere.
\end{proof}

\subsection{Continued fractions}
\label{sec:Continued_Fraction_defs}
Let $\alpha \in \R \backslash \Q$ be an irrational number, and 
\begin{equation*}
    a_0 + \frac{1}{a_1 + \frac{1}{a_2 + \frac{1}{a_3 + \frac{1}{\ddots}}
    }
    }
\end{equation*}
with
\textcolor{black}{$a_0 \in \Z$}, $a_1,a_2,a_3,\dots \in \N_{>0}$ be its continued fraction expansion.
The coefficients $a_i$ are called the {partial quotients} and the fractions of the form
\begin{equation*}
    \frac{p_n}{q_n} = a_0 + \frac{1}{a_1 + \frac{1}{a_2 + \frac{1}{\ddots \frac{1}{a_{n-1} + \frac{1}{a_n}}}
    }
    }
\end{equation*}
are called the \textcolor{black}{convergents} of $\alpha$. We
denote by $\|x\|$ the distance of $x$ to the nearest integer and $\{x\}$ the fractional part of $x$. 

We refer the \textcolor{black}{reader to \cite{Khinchin}}
and \cite[Chapter 3]{einsiedler2010ergodic} for more information on the properties of continued fraction expansions. 
 
We say that an irrational number $\alpha$ is of \textbf{Roth-type} if for all $\epsilon>0$ the growth rate of the partial quotients is of the form $$\textcolor{black}{a_{n+1} = \bigO(q_n^\epsilon).}$$ 
\textcolor{black}{Equivalently}, for every $\epsilon >0,$ there exists a positive constant $C_{\epsilon}$ such that, for all $n\in \N$,
$$q_n^{1+\epsilon}\|q_n\alpha\|>C_{\epsilon}.$$ The set of all Roth-type irrational numbers \textcolor{black}{has} full measure.

We note that \cite{MarmiMoussaYoccoz} generalized this property to IETs with more than two intervals. We do not need the definition of Roth-type $d$-IETs with $d\geq 3$. However, Roth-type 2-IETs are rotations by an irrational number of Roth-type.

There is a residual set of irrational numbers that are not of Roth-type. Let $k \in \N$  be a natural number and denote by  $S_k^0$ the set of irrational numbers $\alpha$ such that the sequence of denominators $(q_n)_{n\in \N}$  \textcolor{black}{ of the continued fraction satisfies}
$$\liminf_{n \to \infty} q_n^{k+1}\|q_n\alpha\|=0.$$

\subsection{Rokhlin towers for rotations}
\label{sec:Rokhlin}
With the properties of continued fractions defined in \Cref{sec:Continued_Fraction_defs}, we can construct a family of Rokhlin towers for the interval $[0,1)$ and the rotation by $\alpha$. \textcolor{black}{At each step, this construction produces two adjacent towers,} one small of width $\|q_{n+1} \alpha \|$ and height $q_n$, and the other tower is wider and very tall of width $\|q_n \alpha\|$ and height $q_{n+1}$. For each $n\in\N$ we have that:\par
\noindent If $n$ is \textcolor{black}{\textbf{odd}}, the \textcolor{black}{bases} of the Rokhlin tower are the intervals $I_{n} = [-\|q_n\alpha\|,0)$ and $\tilde{I}_{n} = [0,\|q_{n+1}\alpha\|)$. The levels of the smaller tower are intervals of the form $R^j\tilde{I}_{n}$ with $j\in\{0,\dots,q_n-1\}$. 
The levels of the larger tower are intervals of the form $R^kI_n$ with $k \in \{0,\dots,q_{n+1}-1\}$.\par

\noindent If $n$ is \textcolor{black}{\textbf{even}}, the \textcolor{black}{bases} are $\tilde{I}_{n} = [-\|q_{n+1}\alpha\|,0)$ and $I_{n} = [0,\|q_{n}\alpha\|)$. The levels of the tower are intervals of the form $R^j\tilde{I}_{n}$ with $j\in\{0,\dots,q_n-1\}$ and $R^kI_n$ with $k \in \{0,\dots,q_{n+1}-1\}$.

\begin{figure}[t]
    \centering
    \begin{subfigure}{0.45\textwidth}
       %\centering
            \begin{tikzpicture}[scale=1]
    \draw[blue,very thick]
    (0,0)--(0.75,0)
    (0,0.5)--(0.75,0.5)
    %(-0.75,1)--(0,1) 
    %(-0.75,1.5)--(0,1.5)
    (0,2)--(0.75,2)
    (0,2.5)--(0.75,2.5);
    \draw[red,very thick] 
    (-3,0)--(0,0)
    (-3,0.5)--(0,0.5)
    (-3,2)--(0,2)
    (-3,2.5)--(0,2.5)
    (-3,3)--(0,3)
    (-3,3.5)--(0,3.5)
    (-3,4)--(0,4)
    (-3,4.5)--(0,4.5)
    (-3,5)--(0,5);
    \node[blue, very thick] at (0.375,1.25) {{\huge $\vdots$}};
    \node[red, very thick] at (-1.5,1.25) {{\huge $\vdots$}};
    \node[very thick] at (0,0)[below] {0};
    \node[very thick] at (0.75,0) [below right] {$\|q_{n+1}\alpha\|$};
    \node[very thick] at (-3,0) [below] {$-\|q_n\alpha\|$};
    \node[very thick] at (-1.5,0) [above] {$x$};
    \node[very thick] at (-1.5,2) [below] {$R^kx$};
    \node[very thick] at (-1.5,5) [above] {$R^{q_{n+1}-1}x$};
    \node[very thick] at (-0.75,0) [below] {$R^{q_{n+1}}x$};
    
    \draw[fill] (-1.5,0) circle[radius=2pt];
    \draw[fill] (-1.5,2) circle[radius=2pt];
    \draw[fill] (-1.5,5) circle[radius=2pt];
    \draw[fill] (-0.75,0) circle[radius=2pt];
\end{tikzpicture}
        \caption{Case $n$ is \textcolor{black}{odd}}
    \end{subfigure}
    \hspace{0.25cm}
    \begin{subfigure}{0.45\textwidth}
        \begin{tikzpicture}[scale=1]
    \draw[blue,very thick]
    (-0.75,0)--(0,0)
    (-0.75,0.5)--(0,0.5)
    (-0.75,2)--(0,2)
    (-0.75,2.5)--(0,2.5);
    \draw[red,very thick] 
    (0,0)--(3,0)
    (0,0.5)--(3,0.5)
    (0,2)--(3,2)
    (0,2.5)--(3,2.5)
    (0,3)--(3,3)
    (0,3.5)--(3,3.5)
    (0,4)--(3,4)
    (0,4.5)--(3,4.5)
    (0,5)--(3,5);
    \node[blue, very thick] at (-0.375,1.25) {{\huge $\vdots$}};
    \node[red, very thick] at (1.5,1.25) {{\huge $\vdots$}};
    \node[very thick] at (0,0)[below] {0};
    \node[very thick] at (-0.75,0) [below left] {$-\|q_{n+1}\alpha\|$};
    \node[very thick] at (3,0) [below right] {$\|q_n\alpha\|$};
    \node[very thick] at (-0.375,0) [above] {$x$};
    \draw[fill] (-0.375,0) circle[radius=2pt];
    \node[very thick] at (-0.375,2.5) [above] {$R^{q_n-1}x$};
    \draw[fill] (-0.375,2.5) circle[radius=2pt];
    \draw[very thick, fill] (3-0.375,0) circle[radius=2pt];
    \node[very thick] at (3-0.375,0) [above] {$R^{q_n}x$};
\end{tikzpicture}
        \caption{Case $n$ is \textcolor{black}{even}}
    \end{subfigure}
    
    \caption{Rokhlin towers for a rotation by $\alpha$.}
    \label{fig:RokhlinTowers}
\end{figure}

In \Cref{fig:RokhlinTowers}, we represent the cases when $n$ is even or odd. The rotation $R$ sends each level to the one immediately above it, except the top levels. When $x$ is in one of the tops of the towers, $R$ sends $x$ to a point in the base. The \textcolor{black}{first return map} to the base of the Rokhlin towers is a rotation by $(-1)^n\|q_{n}\alpha\|$ $\mod \|q_{n+1}\alpha\| + \|q_n\alpha\|$.\par

\begin{remark}
    \label{rmk:shift of towers}
    Notice that by adding $z \bmod{1}$ to the \textcolor{black}{bases} of the towers, we can build Rokhlin towers around $z$, i.e. the \textcolor{black}{bases} now are $[-\|q_n\alpha\|+z,z)$ and $[z,\|q_{n+1}\alpha\|+z)$ or $[-\|q_{n+1}\alpha\|+z,z)$ and $[z,\|q_n\alpha\|+z)$ for $n$ odd or even respectively. This is convenient if we want to avoid discontinuities of the cocycle function inside the levels of the towers.
\end{remark}

\section{Skew products and ergodicity criterion}\label{sec:skprod}
In this section, we recall the main definitions and a few results that we will use later.

\subsection{Essential values of a cocycle}
\textcolor{black}{Let $T:X \to X$ be a map} and $f:X\to \R$ a function.
We denote 
\begin{equation}\label{eq: BSdef}
    S_nf(x) = \left\{
        \begin{array}{lll}
            \sum_{k=0}^{n-1}f(T^kx) & \text{ if }& n\ge 1 \\
             0 & \text{ if } & n=0\\
             -\sum_{k=1}^{|n|}f(T^{-k}x) & \text{ if } & n\le-1 \text{ and $T$ is invertible}
        \end{array}
    \right.
\end{equation}
Let $\mathcal{B}$ be a $\sigma$-algebra on $X$. Suppose that $T:X\to X$ is an invertible ergodic measure-preserving transformation with respect to a finite measure $\mu$ on $(X,\mathcal{B})$. We call such a transformation an {\bf ergodic measure-preserving automorphism}. Consider a measurable function \textcolor{black}{$f:X\to\R$} and the skew product
\[
    T_f:X\times \R \to X\times \R,
    \qquad
    T_f(x,t)=(Tx,t+f(x)).
\]
We say that $a\in\R$ is an {\bf essential value} of $f$ if, for every $B\in\mathcal{B}$ with $\mu(B)>0$ and every $\epsilon>0$, \textcolor{black}{there exists $n\in\Z$} such that
\begin{equation*}
    \mu\left(
        B\cap T^{-n}B \cap
        \left\{x\in X: |S_n f(x)-a|<\epsilon\right\}
    \right)>0.
\end{equation*}

Denote by $\operatorname{Ess}(T_f)$ the set of essential values of $f$. The following holds.
\begin{proposition}[{\cite[Proposition 8.2.1 and Corollary 8.2.5]{aaronson1997}}]
\label{prop:essential_values_imply_ergodicity}
    \textcolor{black}{Suppose that $(X,\mathcal B,\mu)$ is a Borel probability space and that
    $T:X\to X$ is an ergodic measure-preserving automorphism. Let
    $f:X\to \R$ be a measurable function and let
    \[
        T_f:X\times \R\to X\times \R,
        \qquad
        T_f(x,t)=(Tx,t+f(x))
    \]
    be the associated skew product.} Then:
    \begin{enumerate}
    \item $\operatorname{Ess}(T_f)$ is a closed subgroup of $\R$.
    \item $T_f$ is ergodic with respect to $\mu \otimes \Leb_{\R}$ if and only if $\operatorname{Ess}(T_f)=\R. $
    \end{enumerate}
\end{proposition}

We will use the following version of an {\bf essential value criterion}, as introduced by Conze and Frączek.
\begin{theorem}[{\cite[Lemma 2.7]{conze_cocycles_2011}}]
\label{thm:FUcriterion}
\textcolor{black}{Suppose that $(X,d)$ is a compact metric space, $\mu$ is a Borel probability measure on $X$ and $T:(X,\mu) \to (X,\mu)$ is an ergodic measure-preserving automorphism.}
Let $a$ be a real number. Assume that for every $\epsilon>0$ there exists a sequence 
of subsets
$(\Xi_n)_{n\in\N}$ and an increasing sequence 
$(q_n)_{n\in\N}$ of natural numbers such that 
\begin{enumerate}[label=(\arabic*),ref=(\arabic*)]
\item \label{cond:measure} 
$\displaystyle  \liminf_{n\to\infty}\mu(\Xi_n)>0$,
\item \label{cond:rigidity}
$\displaystyle  \lim_{n\to\infty}\sup_{x\in \Xi_n}d(x,T^{q_n}(x))=0$,
\item \label{cond:almost_invariant}
$\displaystyle  \lim_{n\to\infty}\mu(\Xi_n\triangle T(\Xi_n))=0$,
\item  \label{cond:BS} 
$\sup_{x\in\Xi_n}|S_{q_n}f(x)-a|<\epsilon$.
\end{enumerate}
Then $a\in \operatorname{Ess}(T_f)$. 
\end{theorem}
{The result \cite[Lemma 2.7]{conze_cocycles_2011} states that 
the 
topological support of the limit distribution 
$P:=\lim_{n\to\infty}\frac{1}{\mu(\Xi_n)}(S_{q_n}f(x)|_{\Xi_n})_*\mu|_{\Xi_n}$ 
is contained in $\operatorname{Ess}(T_f)$ -- this limit exists up to a subsequence due to tightness guaranteed by  
\cref{cond:BS}. However, again by \cref{cond:BS}, 
the topological support of $P$ is contained in $[a-\epsilon,a+\epsilon]\subset \R$. \textcolor{black}{By 
letting $\epsilon \to 0 $ }and by the fact that $\operatorname{Ess}(T_f)$ is a closed subset
of $\R$, we get $a\in \operatorname{Ess}(T_f)$.
}
\textcolor{black}{Recall that an interval $[0,1)$ can be identified with a circle via the map $r\mapsto e^{2\pi ir}$ and thus we can apply \Cref{thm:FUcriterion} to the systems considered in this article.}

Note that to prove ergodicity of a skew product on the space $[0,1)\times \R$, it is enough to restrict the computation to the induced map on $J\times \R$, where $J \subset[0,1)$ is of positive measure. It follows from the following more general facts. 

\begin{proposition}[{\cite[Proposition 1.5.2]{aaronson1997}}]
\label{prop:PropOfErgodicity}
\textcolor{black}{Suppose that $T:X \to X$ is a conservative map with respect to the measure $m$ on $X$. Let $A \subset X$ be a measurable set with $m(A)>0$. Define the restricted measure $m_A$ by
\[
m_A(C)=m(A\cap C)
\]
for every measurable set $C\subset X$.}
\begin{enumerate}[label=(\roman*),ref=(\roman*)]
\item \label{item:first_cond_induced_map} If $T$ is ergodic with respect to $m$, then the first return map $T_A$ is ergodic with respect to $m_A$.
\item \label{item:second_cond_induced_map} If $T_A$ is ergodic with respect to $m_A$ and
\textcolor{black}{
\[
m \left(X \setminus \bigcup_{n=1}^\infty T^{-n}A\right)=0,
\]
}
then $T$ is ergodic with respect to $m$.
\end{enumerate}
\end{proposition}

As a consequence of \Cref{prop:PropOfErgodicity} and the fact that every skew product with an ergodic base and a zero-mean cocycle is conservative (see \cite[Theorem 5.5]{Schmidt} and \cite[Corollary 8.1.5]{aaronson1997}), we have the following.

\begin{corollary}\label{cor:ErgodicityofSkewProd}
    Let $T:=T_{\lambda,\pi}:I\to I$ be an ergodic $3$-IET with respect to $\mu:=\Leb_I$. Let $f:I\to\R$ be a measurable function with $\int_I f\,\mathrm{d}\mu=0$.
    \textcolor{black}{Let $J\subset I$ be a measurable set with $\mu(J)>0$.}
    Denote by
    \[
        R_{\widetilde f}:J\times\R\to J\times\R
    \]
    the first return map of $T_f$ to the cylinder $J\times\R$. If $R_{\widetilde f}$ is ergodic with respect to $\Leb_J\otimes\Leb_\R$ and
    \[
        \mu\otimes\Leb_\R
        \left(
            (I\times\R)\setminus
            \bigcup_{n=1}^{\infty} T_f^{-n}(J\times\R)
        \right)=0,
    \]
    then $T_f$ is ergodic with respect to $\mu\otimes\Leb_\R$.
\end{corollary}

\section{Proof of  \texorpdfstring{\Cref{thm:main1}}{}}\label{sec:mainproof}

\subsection{Reducing \texorpdfstring{$T_f$}{} to \texorpdfstring{$R_{\widetilde f}$}{}}

To make the statement of \Cref{thm:main1} more precise, let $T$ be a minimal $3$-IET and let $f\in C^{1+AC}([0,1))$ satisfy
\begin{equation}
    f(0)=\lim_{x\to 1}f(x)=0.
    \label{eq:LimitCondThm1}
\end{equation}
As \textcolor{black}{explained in Cases 1 and 2 (see \Cref{eq:case1,eq:case2}), the first return map of the skew product $T_f$ to $J\times\R$ is of the form} 
$R_{\widetilde f},$
where $R$ is a rotation on $J$ and $\widetilde f$ is a function with at most one discontinuity in the interior of $J$. In Case 1, this discontinuity occurs at $\lambda_A-\lambda_C$, while in Case 2 there are no discontinuities in the interior of $J$. This follows from \Cref{eq:LimitCondThm1}.

We note that every $3$-IET is measurably isomorphic to its inverse. More precisely, the map
\[
    \mathcal I(x)=|I|-x
\]
gives an isomorphism between $T$ and $T^{-1}$. Hence $T_f$ is measurably isomorphic to $(T^{-1})_{-f\circ\mathcal I}$, with isomorphism
\[
    (x,t)\mapsto (\mathcal I(x),-t).
\]
\textcolor{black}{Consequently, to disprove ergodicity of $T_f$, it is enough to consider Case 2, while to prove ergodicity of $T_f$, it is enough to consider Case 1.}

If $f$ has bounded variation, the induced cocycle $\widetilde{f}$ in \Cref{eq:case2} also has bounded variation. Similarly, $\widetilde{f}$ is absolutely continuous if $f$ is absolutely continuous and satisfies \Cref{eq:LimitCondThm1}, see \Cref{rmk:ftildeisabscont}. 
\textcolor{black}{
Recall that under \Cref{eq:case1,eq:LimitCondThm1},
the induced cocycle $\widetilde f$ is only piecewise absolutely continuous,
not necessarily absolutely continuous on the whole interval. However, the
restrictions of $\widetilde f$ to the two intervals exchanged by $R$ are
absolutely continuous.
}

At this point, we consider the skew product $R_{\widetilde{f}}$ where $R$ is a rotation on the interval $[0,1)$ by $\alpha \in \R\backslash\Q$ and
\begin{enumerate}
    \item $\widetilde{f}$ is an absolutely continuous function,
    \item the derivative of $\widetilde{f}$ has bounded variation,
    \item $\int\widetilde{f} = 0$ and $\int\widetilde{f}' = 0$.
\end{enumerate}

Because of the above reduction and \Cref{rmk:inducedenough}, the proof of \Cref{thm:main1} boils down to showing the following fact.

\begin{proposition}
\label{prop:R_f_is_not_ergodic}
Let $R$ be the rotation by an irrational Roth-type number $\alpha$. Let
$\widetilde f:[0,1)\to\R$ be a function such that $\int \widetilde f\,dx=0$.
\textcolor{black}{Assume that the restrictions of $\widetilde f$ to the intervals
$[0,1-\alpha)$ and $[1-\alpha,1)$ are absolutely continuous, that the corresponding restrictions of
$\widetilde f'$ have bounded variation, and that $\int \widetilde f'\,dx=0$.
}
Then $\widetilde f$ is a coboundary for $R$. In particular, the skew product
$R_{\widetilde f}$ is not ergodic.
\end{proposition}

\subsection{The solution of the cohomological equation}
The proof of \Cref{prop:R_f_is_not_ergodic}
follows from:

\begin{theorem}[{\cite[Theorem A]{MarmiMoussaYoccoz}}]
\label{thm:MarmiMoussaYoccoz}
    Let $T$ be an IET with the Keane property and of Roth-type. Let $\Phi \in BV^1_*\left(\bigsqcup_{s\in\mathcal{A}} I_s\right)$. There exists a function $\chi$ constant on each interval $I_s$ and a bounded function $\Psi$ such that 
    \begin{equation}
    \Psi \circ T - \Psi = \Phi - \chi.
\label{eq:MarmiMoussaYoccoz_Equation}
    \end{equation}
\end{theorem}

For context, $\Phi \in BV^1_*(\bigsqcup I_s)$ if $\Phi|_{I_s}$ is absolutely continuous, $\Phi'|_{I_s}$ has bounded variation and $\int \Phi'=0.$
\begin{proof}[Proof of \Cref{prop:R_f_is_not_ergodic}]

Take $T =  R$, where $R$ is the rotation by the Roth-type irrational number $\alpha$. Take $\Phi = \widetilde{f}$. \textcolor{black}{Note that 
\[
\int_J\widetilde{f}'=\int_I f'=\lim_{x\to 1}f(x)-f(0)=0.
\]
} 
By \Cref{thm:MarmiMoussaYoccoz}, there exist a piecewise constant function
\[
\chi=a_1\mathds{1}_{[0,1-\alpha)}+a_2\mathds{1}_{[1-\alpha,1)}
\]
and \textcolor{black}{a bounded function $\Psi$} such that
\[
\Psi\circ R-\Psi=\widetilde f-\chi.
\]
Since $\int \widetilde{f} = 0$, the previous equation implies that $\int \chi = 0$. Hence, we can assume that \textcolor{black}{$a_1=c\alpha$ and $a_2=-c(1-\alpha)$ for some $c \in \R$.}
    Notice that the function $\chi = c\alpha \mathds{1}_{[0, 1-\alpha)} +c(\alpha - 1)\mathds{1}_{[1-\alpha, 1)}$ is a coboundary, that is, $h\circ R - h = \chi$ for $h(x) := cx$:
    \begin{equation*}
        h(x + \alpha \bmod{1}) - h(x) = \left\{\begin{array}{lc}
            c(x + \alpha - x) & x\in[0,1-\alpha), \\
            c(x-(1-\alpha) - x) & x\in[1-\alpha,1). 
        \end{array}\right.
    \end{equation*}

    It follows that $\widetilde{f}$ is a coboundary:
    \begin{equation*}
        \widetilde{f}= (\Psi + h)\circ R - (\Psi + h)
    \end{equation*}
    and $R_{\widetilde{f}}$ is not ergodic because sets of the form $\{(x,s):s \leq \Psi(x) + h(x)\}$ are invariant.
\end{proof}

\textcolor{black}{
In order to apply \Cref{thm:MarmiMoussaYoccoz}, it is not necessary to assume that $f’$ is absolutely continuous; it is enough to assume that $f’$ has bounded variation and that $\int f’=0$. 
However, this weaker hypothesis is not sufficient for the second part of \Cref{thm:main1}, as the next result shows.
}

\subsection{Ergodic examples and the concluding argument}

A function $g:\T \to \R$ is said to be piecewise absolutely continuous if there are $\beta_0 < \dots < \beta_k \in \T$ (finitely many discontinuities) such that for all $j=0,\dots,k$, the restriction of $g$ to the interval $(\beta_j,\beta_{j+1})$ is absolutely continuous, where we set $\beta_{k+1} := \beta_0$. Denote the jump of $g$ at the discontinuity $\beta_j$ by
\[
a_j = \lim_{x \to \beta_j^+}g(x) - \lim_{y \to \beta_j^-}g(y).
\]
\textcolor{black}{
Denote by $S(g)$ the sum of all its jumps and observe that
\[
S(g) := \sum_{j=0}^ka_j = -\int g'.
\]
The last equality follows from the circle convention and the definition of the jumps.
}

We say that $g \in C_0^{1+\textrm{BV}}$ if $g:\T \to \R$ is absolutely continuous with $\int g =0$ and $g'$ has bounded variation. We say that $g \in C_0^{1+\textrm{PAC}}$ if $g \in  C_0^{1+\textrm{BV}}$ and $g'$ is piecewise absolutely continuous.

Fr{\k a}czek in \cite{Krzysztof2000} proved the  following \namecref{thm:Krzysztof} in greater generality. We only state a simpler version to prove the second part of \Cref{thm:main1}.
\begin{theorem}[{\cite[Theorem 1.1]{Krzysztof2000}}]
\label{thm:Krzysztof}
Let $R$ be a rotation by $\alpha \in S_1^0$ and let $g \in C_0^{1+\textrm{PAC}}$ be such that $S(g')=0$.
Let $\beta_0,\dots,\beta_k\in \T$ \textcolor{black}{
be the discontinuities of $g'$, namely, 
\[
a_j\neq 0
\quad 
\text{for every}
\quad
j=0,\ldots,k.
\]
}
If there exists a subsequence $q_{n_1}<q_{n_2}<q_{n_3}<\dots$ of denominators of the continued fraction of $\alpha$ such that 
\begin{equation}
\label{eq:Krzysztof}
    \lim_{j\to \infty}q_{n_j}^2\|q_{n_j}\alpha\|=0 \quad \text{ and } \quad \lim_{j\to \infty} \{q_{n_j}\beta_i\} = \gamma_i,
\end{equation}
where $\gamma_i \neq \gamma_l$ for $i\neq l$, $i,l =0,\dots,k,$ then $R_g$ is ergodic.
\end{theorem}

\begin{proof}[Proof of \Cref{thm:main1}]

We first prove the first part of the statement. The full-measure set of
$3$-IETs for which $f$ is a coboundary follows directly from
\Cref{prop:R_f_is_not_ergodic} and \Cref{rmk:inducedenough}. Indeed, for a
full-measure set of parameters $(\lambda_A,\lambda_B,\lambda_C)$, the
rotation obtained by inducing on $[0,\lambda_B+\lambda_C)$, after
normalization, is of Roth-type. In the other case, we apply the same argument
to $T^{-1}$.

We now prove the second part of the statement. Let $T=T_{\lambda,\pi}$ be a
$3$-IET, and let $f:[0,|\lambda|)\to\R$ be a
$C_0^{1+AC}([0,|\lambda|))$-function such that
\[
    f(0)=\lim_{x\to |\lambda|}f(x)=0.
\]
According to \Cref{rmk:ftildeisabscont}, when Case~1, \Cref{eq:case1},
applies, the induced cocycle $\widetilde f:J\to\R$ is piecewise absolutely
continuous and has a possible nonzero jump at $\lambda_A-\lambda_C$.
Denote this jump by
\[
c:=
\lim_{x\to(\lambda_A-\lambda_C)^+}\widetilde f(x)
-
\lim_{x\to(\lambda_A-\lambda_C)^-}\widetilde f(x).
\]
Recall that $R:J\to J$ is the rotation by $\lambda_B+\lambda_C$. Define
\[
h(x):=\frac{c}{|J|}x
\qquad \text{and} \qquad
\widehat f:=\widetilde f+h\circ R-h.
\]
Then $h\circ R-h$ has jump $-c$ at $\lambda_A-\lambda_C$. Hence the jump
of $\widetilde f$ at $\lambda_A-\lambda_C$ is cancelled, and $\widehat f$
is absolutely continuous. Moreover,
\[
    \int_J\widehat f=\int_J\widetilde f=0,
\]
because $h\circ R-h$ has integral zero. Hence $\widehat f\in C_0^{1+PAC}$.

Assume now that
\[
f'(0)=\lim_{x\to |\lambda|}f'(x).
\]
Then
\[
S(f')=S(\widetilde f')=S(\widehat f')=0.
\]
If, in addition, $f'$ does not vanish at the discontinuities of $T$, then
$\widehat f'$ has a discontinuity at
$x=\lambda_A\in(\lambda_A-\lambda_C,\lambda_A+\lambda_B)$. Indeed,
by assumption $\lim_{x\to 1^-}f'(x)\neq 0$, and
\[
\lim_{x\to \lambda_A^+}\widehat f'(x)
=
f'(\lambda_A)
\neq
f'(\lambda_A)+\lim_{x\to 1^-}f'(x)
=
\lim_{x\to \lambda_A^-}\widehat f'(x).
\]

The condition in \Cref{eq:Krzysztof} is a dense condition on the length
parameter $\lambda$. After normalizing so that $|J|=1$, \Cref{thm:Krzysztof}
applied to $\widehat f$ implies that $R_{\widehat f}$ is ergodic. Since
$\widehat f$ and $\widetilde f$ differ by a coboundary, the skew products
$R_{\widehat f}$ and $R_{\widetilde f}$ are measurably isomorphic. Hence
$R_{\widetilde f}$ is ergodic. By \Cref{cor:ErgodicityofSkewProd}, $T_f$ is
ergodic for a dense set of parameters $\lambda$.
\end{proof}

\subsection{Distinct endpoint values}
\textcolor{black}{%
To place our results in a broader context, we provide the result under the assumption that the values of $f$ at the endpoints do not coincide.\\
Recall that if $f$ is piecewise absolutely continuous, then so is the induced cocycle $\widetilde f$. \textcolor{black}{Moreover, the sum of the jumps of $f$ is equal to the sum of the jumps of $\tilde f$.} Pask in \cite[Theorem 2.11]{Pask} proved that for every irrational rotation $R$ and piecewise absolutely continuous function $\widetilde f$ with zero mean and nonzero total sum of jumps, the skew product $R_{\widetilde f}$ is ergodic.
Hence, using the reduction to the induced rotation described in \Cref{sec:reduction_to_rotation} and \Cref{cor:ErgodicityofSkewProd}, we obtain the following.
\begin{proposition}\label{prop:Pask}
    For every ergodic 3-IET $T$ on $[0,1)$, if $f\in \textcolor{black}{C^{1+AC}([0,1))}$ is such that \textcolor{black}{$\int f=0$} and $f(0)\neq \lim_{x\to 1}f(x)$, then the skew product $T_f$ is ergodic.
\end{proposition}
}

\section{Proof of \texorpdfstring{\Cref{thm:main2}}{}}\label{sec:mainproof2}
In this section, we prove  \Cref{thm:main2}. Using the essential value criterion, we will prove that skew products are ergodic (\Cref{thm:FUcriterion}).
The primary tool will be the discontinuity that appears in the interior (not in the orbit of 0) of one of the intervals exchanged by $T|_J$, which, in our case, is a rotation. 

\subsection{Initial reductions.} 

\textcolor{black}{
Assume that
\[
f\in C^{1+AC}([0,1)),\qquad
\int f=0,
\qquad
f(0)=\lim_{x\to 1^-}f(x)\neq 0.
\]
}
\textcolor{black}{To simplify the computations, after multiplying $f$ by a scalar if necessary, we can assume that $\lim_{x\to 1^-}f(x)=1$.}

Without loss of generality, we assume that Case 1 holds, that is, \Cref{eq:case1}. In this case, the induced cocycle $\widetilde f$, viewed as a function on the circle, has a guaranteed discontinuity at $\lambda_A$. \textcolor{black}{We do not claim that $\widetilde f$ necessarily has discontinuities at $0$ or at $\lambda_A-\lambda_C$.}

We now reverse the problem. Namely, we consider almost every $\alpha\in[0,1)$, almost every $\beta\in[0,1)$, and a zero-mean cocycle $f$ over the rotation $R(x)=x+\alpha$, where $f$ is piecewise $C^{1+AC}$. Specifically, $f$ is $C^{1+AC}$ on
\[
[0,1-\alpha),\qquad [1-\alpha,\beta),\qquad [\beta,1),
\]
\textcolor{black}{
and has a nontrivial jump at $\beta$, that is,
}
\[
\lim_{x\to\beta^+}f(x)-\lim_{x\to\beta^-}f(x)\neq 0.
\]
As above, after multiplying by a nonzero scalar, we may assume that this jump is equal to $-1$. We devote the rest of this section to proving that, under these assumptions, the skew product $R_f$ is ergodic.

\begin{figure}[ht]
    \centering
    \includegraphics[width=0.75\linewidth, trim={1.4cm 0cm 0cm 0cm}, clip=true]{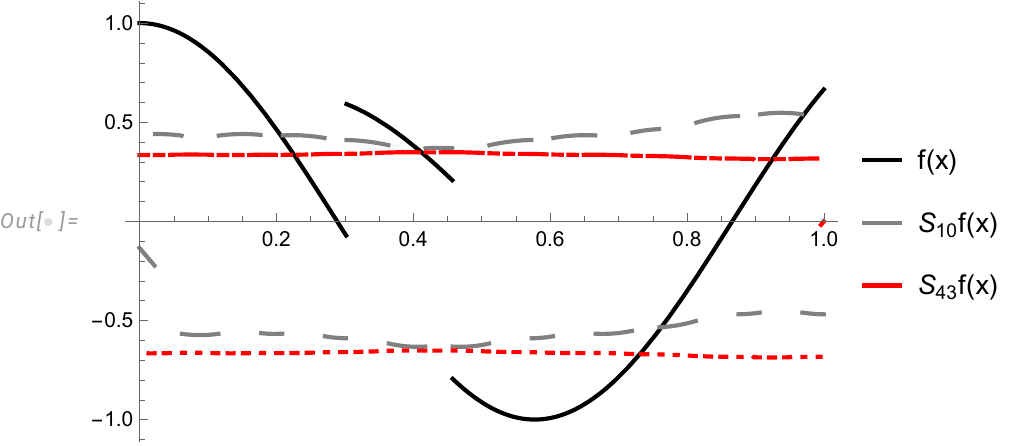}
    \caption{Behavior of the Birkhoff sums $S_k f$ for different values of $k$. The cocycle $f$ has possible discontinuities at $0$, $1-\alpha$ and $\beta$, with a nontrivial jump at $\beta$. The argument used to prove ergodicity in \Cref{thm:main2} \textcolor{black}{consists in estimating} the values of the jumps at specific times $k$ \textcolor{black}{and combining this} with \Cref{thm:FUcriterion}.}
    \label{fig:Cosine_example}
\end{figure}

\subsection{Bounds of the Birkhoff sums.} Let us start by recalling the standard Denjoy-Koksma inequality. Namely, if $F:[0,1)\to\R$ has bounded variation, then
\[
\left\|S_{q_n}F-q_n\int F \right\|_{\infty} \le Var(F),
\]
\textcolor{black}{where $(q_n)_{n\in\N}$ is a sequence of denominators of the convergents of $\alpha$.}
In particular, in our case, we have
\begin{equation}\label{eq: DKfunction}
    \|S_{q_n}f\|_{\infty}\le Var(f) 
\end{equation}
and
\begin{equation}\label{eq: DKderivative}
    \|S_{q_n}f'\|_{\infty}\le Var(f'),
\end{equation}
\textcolor{black}{where in the latter we used the fact that $\int_0^1 f'=\lim_{x\to1^-}f(x)-f(0)=0$.}

From now on, we assume that $\alpha$ is of unbounded type, \textcolor{black}{meaning that if $(a_n)_{n\in\N}$  is the }sequence of partial quotients, then there exists a subsequence $(n_k)_{k\in\N}$ such that 
\begin{equation}\label{eq:unboundedquotients}
    \lim_{k\to\infty} a_{n_k+2}=\infty.
\end{equation}

We begin by showing that with the above assumption, the integrals of $S_{q_n}f$ converge quickly to 0 on properly chosen intervals.
\begin{lemma}\label{lem: smallintegrals}
\textcolor{black}{Assume that $f$ and $f'$ have bounded variation, and \Crefrange{eq: DKfunction}{eq:unboundedquotients} hold.} Then there exists a sequence $(c_k)_{k\in\N}$ of positive numbers such that
\[
\sup\left\{\left|\int_{J}S_{q_{n_k+1}}f\right| : J\subset [0,1) \text{ with }\ |J|=\|q_{n_k}\alpha\|\right\}\le c_k= \bigO(\|q_{n_k+1}\alpha\|).
\] 
\end{lemma}
\begin{proof}
Recall that for every $n\in\N$, the first return map to any interval $L=[c,d)\subset [0,1)$ of length $\|q_{n+1}\alpha \|+\|q_n\alpha\|$ is a 2-IET (a rotation) and the discontinuity point is either $c+\|q_{n+1}\alpha \|$ or $d-\|q_{n+1}\alpha \|$. For further details, see \Cref{sec:Rokhlin}. In both cases, the entire interval splits into two adjacent Rokhlin towers of intervals (seen on the circle), where the one with the shorter base has height $q_n$ and the one with the longer base has height $q_{n+1}$. Without loss of generality, we assume that the former holds for all \textcolor{black}{considered} $n\in\N$, \textcolor{black}{since the other case is symmetric}.

Fix $k\in\N$ and let $J=[a,b)$ as in the assumption of the lemma. Consider an interval $J':=[a-\|q_{n_k+1}\alpha\|,b)$. Then the full circle can be decomposed into two towers, one over the interval $\widehat J=[a-\|q_{n_k+1}\alpha\|,a)$ of height $q_{n_k}$ and the other over the interval $J$ of height $q_{n_k+1}$.
\textcolor{black}{
Since $\int f=0$, we have 
\begin{equation}
\label{eq:integrals_of_sums}
0  = \int f = \int_{J} S_{q_{n_k+1}}f + \int_{\widehat{J}}S_{q_{n_k}}f.
\end{equation}
}
By the Denjoy-Koksma inequality,
\[
\left|\int_{\widehat J} S_{q_{n_k}}f\right|\le \|q_{n_k+1}\alpha\|\cdot Var(f).
\]
Combining this with \Cref{eq:integrals_of_sums}, we obtain
\[
\left|\int_{J} S_{q_{n_k+1}}f\right|\le \|q_{n_k+1}\alpha\|\cdot Var(f).
\]
Thus, it suffices to take $c_k:=\|q_{n_k+1}\alpha\|\cdot Var(f)$.
\end{proof}

\subsection{\texorpdfstring{$(\delta,\infty)$}{}-balanced points}

For every $k\in\N$, consider the Rokhlin towers with \textcolor{black}{bases given by the} intervals 
\begin{equation}
\label{eq:FixedCase}
I_k^1=[0,\|q_{n_k+1}\alpha\|)
\quad\text{and}\quad
I_k^2=[\|q_{n_k+1}\alpha\|,\|q_{n_k}\alpha\|+\|q_{n_k+1}\alpha\|)
\end{equation}
or
\[
I_k^1=[0,\|q_{n_k}\alpha\|)
\quad\text{and}\quad
I_k^2=[\|q_{n_k}\alpha\|,\|q_{n_k}\alpha\|+\|q_{n_k+1}\alpha\|),
\]
depending on the parity of $n_k$; see \Cref{rmk:shift of towers}. Without loss of generality, by passing to a subsequence if necessary, we may assume that only one of the two alternatives above occurs in the Rokhlin tower decomposition. We assume that this is the case specified by the intervals in \eqref{eq:FixedCase}. The argument in the other case is analogous.

Let $\delta\in(0,1)$. We say that $z\in[0,1)$ is \textbf{$(\delta,\infty)$-balanced} if \textcolor{black}{there exist a subsequence $(n_{k_\ell})_{\ell\in\N}$ and a sequence of positive numbers $(\epsilon_\ell)_{\ell\in\N}$ with $\epsilon_\ell\to0$ such that, writing}
\[
I_{k_\ell}^2=[a_{k_\ell},b_{k_\ell})
:=
[\|q_{n_{k_\ell}+1}\alpha\|,
\|q_{n_{k_\ell}}\alpha\|+\|q_{n_{k_\ell}+1}\alpha\|),
\]
we have
\[
z\in
\bigcup_{j=0}^{q_{n_{k_\ell}+1}-1}
R^j\left(
(\delta-\epsilon_\ell)b_{k_\ell},
(\delta+\epsilon_\ell)b_{k_\ell}
\right)
\]
for every $\ell\in\N$.

It turns out that $\Leb$-almost every point of $[0,1)$ is $(\delta,\infty)$-balanced.

\begin{lemma}\label{lem:balanced}
For every $\delta\in(0,1)$, $\Leb$-almost every $z\in[0,1)$ is $(\delta,\infty)$-balanced.
\end{lemma}

\textcolor{black}{\Cref*{lem:balanced} is a direct corollary of the following \namecref{lem:FayadLemanczyk},} whose proof we omit. The proof is identical to the proof of \cite[Lemma~4]{FayadLemanczyk}.

\begin{lemma}\label{lem:FayadLemanczyk}
Fix $\delta\in(0,1)$. For every $k\in\N$, let $\epsilon_k=k^{-1/2}$ and consider the sets
\[
C_k:=
\bigcup_{j=0}^{q_{n_k+1}-1}
R^j\left(
(\delta-\epsilon_k)b_k,
(\delta+\epsilon_k)b_k
\right).
\]
\textcolor{black}{Then
\[
\Leb\left([0,1)\setminus\limsup_{k\to\infty} C_k\right)=0.
\]}
\end{lemma}

\subsection{Main estimates and concluding arguments.} We now go back to the proof of \Cref{thm:main2}. Let $\beta\in [0,1)$ be such that it satisfies the conclusion of \Cref{lem:balanced} with $\delta=\tfrac{1}{\sqrt{2}}$, i.e. $\beta$ is $\left(\tfrac{1}{\sqrt{2}},\infty\right)$-balanced. \textcolor{black}{To simplify notation}, we write that $\beta$ is $\left(\tfrac{1}{\sqrt{2}},\infty\right)$-balanced with the sequences $(n_k)_{k\in\N}$ and $(\epsilon_k)_{k\in\N}$. Moreover, let $(c_k)_{k\in\N}$ be the sequence given by \Cref{lem: smallintegrals}. Then for every $k\in\N$ and every $0\le j<q_{n_k+1}$, we have
\begin{equation}\label{eq:smallintegrals}
\left|
\int_{R^j
[\|q_{n_k+1}\alpha \|
,
\|q_{n_k}\alpha\|
+
\|q_{n_k+1}\alpha \|)
}
S_{q_{n_k+1}}f
\right|
\le 
c_k = \bigO(\|q_{n_k+1}\alpha\|).
\end{equation}
To finish the proof, we will find two sequences of subsets $(\Xi_k^1)_{k\in\N}$ and $(\Xi_k^2)_{k\in\N}$ of $[0,1)$ which satisfy \Cref{thm:FUcriterion} for $a$ equal to $1-\tfrac{1}{\sqrt{2}}$ and $-\tfrac{1}{\sqrt{2}}$, respectively. 

For any $k\in\N$, define $0\le m_k<q_{n_k+1}$ to be such that 
\[
\beta\in R^{m_k}[\|q_{n_k+1}\alpha\|,\|q_{n_k}\alpha\|+\|q_{n_k+1}\alpha\|).
\]
Let
\[
d_j^k=R^j(R^{-{m_k}}\beta) 
\quad \text{for} 
\quad 0\le j\le m_k
\]
and
\[
d_j^k=R^j(R^{-{m_k}}\beta+\|q_{n_k+1}\alpha\|) 
\quad \text{for} 
\quad 
m_k<j<q_{n_k+1}.
\]
Then for every $0 \le j < q_{n_k+1}$
we define
\[
J_j^1:=[R^{j}\|q_{n_k+1}\alpha \|,d_j^k)
\qquad 
\text{and}
\qquad J_j^2:=[d_j^k,R^{j}(\|q_{n_k}\alpha\|+\|q_{n_k+1}\alpha \|)).
\]
Note that $S_{q_{n_k+1}}f$ is continuous on every interval of the form $J_j^1$ and $J_j^2$.
Now we take
\begin{equation}\label{eq: defofXi}
\Xi_k^1:=\bigcup_{j=0}^{q_{n_k+1}-1}J_j^1
\qquad 
\text{and}
\qquad \Xi_k^2:=\bigcup_{j=0}^{q_{n_k+1}-1}J_j^2.
\end{equation}
By the choice of $\delta$, we have
\[
\min\left\{
\liminf_{k\to\infty} \Leb(\Xi_k^1),\ \liminf_{k\to\infty} \Leb(\Xi_k^2)
\right\}>0.
\]
Since $(q_{n_k+1})_{k\in\N}$ is a subsequence of denominators of the convergents of $\alpha$, we have the limits
\[
\lim_{k\to\infty} \sup |R^{q_{n_k+1}}(x)-x|=\lim_{k\to\infty}\|q_{n_k+1}\alpha\|=0.
\]
Moreover, since both $\Xi_k^1$ and $\Xi_k^2$ are unions of at most two towers with widths smaller than $\|q_{n_k}\alpha\|$, we have
\[
\lim_{k\to\infty}\Leb(\Xi_k^1\triangle R\Xi_k^1)
=
\lim_{k\to\infty}\Leb(\Xi_k^2\triangle R\Xi_k^2)=0.
\]
Hence, it remains to check whether condition \ref{cond:BS} of \Cref{thm:FUcriterion} is satisfied with $a=1-\tfrac{1}{\sqrt{2}}$ for $(\Xi_k^1)_{k\in\N}$ and with $a=-\tfrac{1}{\sqrt{2}}$ for $(\Xi_k^2)_{k\in\N}$. 

Fix $\epsilon>0$. Take $k\in\N$ and $0\le j\le q_{n_k+1}-1$. In view of \eqref{eq:smallintegrals}, we have 
\[
\Big|\int_{J_j^1} S_{q_{n_k+1}}f\ +\ \int_{J_j^2} S_{q_{n_k+1}}f\Big| \le c_k.
\]
Let $\omega_j^-:=\lim_{x\to {d_j^k}^-} S_{q_{n_k+1}}f(x)$ and $\omega_j^+:=\lim_{x\to {d_j^k}^+} S_{q_{n_k+1}}f(x)$. Since $\beta$ is chosen to be the only discontinuity of $f$ in the interior of the tower 
\[
\bigcup_{i=0}^{q_{n_k+1}-1}R^{j+i}[\|q_{n_k+1}\alpha\|,\|q_{n_k+1}\alpha\|+\|q_{n_k}\alpha\|),
\]
(the other possible discontinuities are $0$ and $R^{-1}(0)=1-\alpha$) and the jump in $\beta$ is $-1$, we have
\begin{equation}
\omega_j^+-\omega_j^-=-1\ \text{for every }0\le j<q_{n_k+1}.
\end{equation}
Since for every $0\le j<q_{n_k+1}$, the function $S_{q_{n_k+1}}f$ is continuous on both $J_j^1$ and $J_j^2$, \textcolor{black}{by the mean value theorem we get for every} $x\in J_j^1$ that
\begin{equation}\label{eq: MVT1}
|S_{q_{n_k+1}}f(x)-\omega_j^-|\le \|q_{n_k}\alpha\|\cdot\sup_{y\in J_j^1}|S_{q_{n_k+1}}f'(y)|\le \|q_{n_k}\alpha\|\cdot Var(f'),
\end{equation}
and for every $x\in J_j^2$
\begin{equation}\label{eq: MVT2}
|S_{q_{n_k+1}}f(x)-\omega_j^+|\le \|q_{n_k}\alpha\|\cdot\sup_{y\in J_j^2}|S_{q_{n_k+1}}f'(y)|\le \|q_{n_k}\alpha\|\cdot Var(f').
\end{equation}
Thus we have 
\[
|J_j^1|(\omega_j^--\|q_{n_k}\alpha\|\cdot Var(f'))\le \int_{J_j^1} S_{q_{n_k+1}}(f)
\le |J_j^1|(\omega_j^-+\|q_{n_k}\alpha\|\cdot Var(f'))
\]
and 
\[
|J_j^2|(\omega_j^+-\|q_{n_k}\alpha\|\cdot Var(f'))\le \int_{J_j^2} S_{q_{n_k+1}}(f)
\le |J_j^2|(\omega_j^++\|q_{n_k}\alpha\|\cdot Var(f')).
\]
By combining the two above inequalities, we get
\begin{equation}
\label{eq:upperintbound}
|J_j^1|\omega_j^-+|J_j^2|\omega_j^+\le c_k+2\|q_{n_k}\alpha\|^2\cdot Var(f')
\end{equation}
and
\begin{equation}\label{eq:lowerintbound}
|J_j^1|\omega_j^-+|J_j^2|\omega_j^+ \ge -c_k-2\|q_{n_k}\alpha\|^2\cdot Var(f').
\end{equation}

Take $\varepsilon>0$. Since $c_k$ is of order $\bigO(\|q_{n_k+1}\alpha\|)$ and $\lim_{k\to\infty} a_{n_k+1}=\infty$, $c_k$ is of order $\smallo(\|q_{n_k}\alpha\|)$. Moreover, since $\beta$ is $(\tfrac{1}{\sqrt{2}},\infty)$-balanced, by making $k$ larger if necessary, for every $0\le j<q_{n_k+1}$ we have $|J_j^1|\ge \frac{1}{2}\|q_{n_k}\alpha\|$. Making $k$ larger again if necessary, \eqref{eq:upperintbound} and \eqref{eq:lowerintbound} together yield
\[
\left|\, \omega_j^-+\frac{|J_j^2|}{|J_j^1|}\omega_j^+ \right|\le \varepsilon.
\]
Moreover, since $\omega_j^+-\omega_j^-=-1$, we get
\begin{equation}\label{eq: omegainequality}
\left| \left(1+\frac{|J_j^2|}{|J_j^1|}\right)\omega_j^--\frac{|J_j^2|}{|J_j^1|}  \right|\le \varepsilon.
\end{equation}
Since 
\[
\lim_{k\to\infty}\epsilon_k=0=\lim_{k\to\infty}\tfrac{\|q_{{n_k}+1}\alpha\|}{\|q_{{n_k}}\alpha\|},
\]
by making $k$ larger again if necessary, it follows from $\beta$ being $(\frac{1}{\sqrt{2}},\infty)$-balanced that
\[
\max_{0\le j<q_{n_k+1}}\left|\frac{|J_j^2|}{|J_j^1|}-(\sqrt{2}-1)\right|\le \varepsilon.
\]
By combining this inequality with \eqref{eq: omegainequality}, we get
\[
\frac{\sqrt{2}-1-2\varepsilon}{\sqrt{2}+\varepsilon}\le \omega_j^-\le
\frac{\sqrt{2}-1+2\varepsilon}{\sqrt{2}-\varepsilon}.
\]
By decreasing $\varepsilon$ and increasing $k$ again if necessary, we obtain from the above inequality and from \eqref{eq: MVT1} that:
\[
S_{q_{n_k+1}}(f)(x)\in (1-\tfrac{1}{\sqrt{2}}-\varepsilon,1-\tfrac{1}{\sqrt{2}}+\varepsilon)\ \text{for every }x\in\Xi_{k}^1.
\]
Moreover, since $\omega_j^+-\omega_j^-=-1$, we get from \eqref{eq: MVT2} that 
\[
S_{q_{n_k+1}}(f)(x)\in (-\tfrac{1}{\sqrt{2}}-\varepsilon,-\tfrac{1}{\sqrt{2}}+\varepsilon)\ \text{for every }x\in\Xi_{k}^2.
\]
This proves condition \ref{cond:BS} in  \Cref{thm:FUcriterion} for $(\Xi_k^1)_{k\in\N}$ with $a=1-\frac{1}{\sqrt{2}}$ and $(\Xi_k^2)_{k\in\N}$ with $a=-\frac{1}{\sqrt{2}}$. Hence, the assumptions of \Cref{thm:FUcriterion} are satisfied for both sequences. Since $1-\frac{1}{\sqrt{2}}$ and $-\frac{1}{\sqrt{2}}$ generate a dense subgroup of $\R$
and since $\operatorname{Ess}(R_f)$ is a closed subgroup of $\R$, \Cref{prop:essential_values_imply_ergodicity} implies that $R_f$ is ergodic. \Cref{cor:ErgodicityofSkewProd} then gives the ergodicity of the original skew product, completing the proof of \Cref{thm:main2}.

\subsection{Non-ergodic examples and proof of \texorpdfstring{\Cref*{prop:main3}}{}.} 
\textcolor{black}{%
We finish this section by proving \Cref{prop:main3}. Using the reduction given by Cases 1 and 2, namely, \Cref{eq:case1,eq:case2}, we first consider the induced skew product over a rotation. Assume that the rotation $R$ on $[0,1)$,
obtained by Rauzy-Veech induction from the $3$-IET $T$, 
is of Roth-type and that the discontinuity $\beta$ of $\widetilde f$ is at $R^{-m}(0)$ for some $m\ge 1$. Then $\widetilde f$ is piecewise absolutely continuous. Since the orbit of $0$ is dense for every irrational rotation, the set of $3$-IETs we consider is dense in the space of all $3$-IETs. \\
Let $n\ge 0$ be such that $q_n\ge m$. Then, by taking $J:= [0,\|q_n\alpha\|+\|q_{n+1}\alpha\|)$, we have that the induced 2-IET $R|_J$ is of Roth-type. Denote by $\widehat f(x):=S_{r(x)}\widetilde f(x)$ the induced cocycle function where $r:J\to \N$ is the first return time to $J$. 
Moreover, $\widehat f$ is piecewise absolutely continuous, has zero mean, its derivative has bounded variation on each
interval of continuity of $R|_J$, and the integral of its derivative over these
intervals is zero.
}
Therefore, 
\Cref{prop:R_f_is_not_ergodic} applies and shows that $\widehat f$ is a coboundary for $R|_J$. Recall that  \Cref{rmk:inducedenough} shows that $f$ is a coboundary for $T$.

\bibliography{references}
\bibliographystyle{alpha}

\end{document}